\documentclass[11pt]{article}
\usepackage{amssymb,amsmath,amsthm,graphicx}
\usepackage{float}
\usepackage{color,enumerate}
\usepackage{fullpage}
\usepackage{hyperref}
\usepackage{cite}
\usepackage{amssymb}
\usepackage{graphicx}

\newtheorem{theorem}{Theorem}[section]
\newtheorem{lemma}[theorem]{Lemma}
\newtheorem{corollary}[theorem]{Corollary}
\newtheorem{proposition}[theorem]{Proposition}

\newtheorem*{problem}{Open problem}

\theoremstyle{definition}

\newtheorem{definition}[theorem]{Definition}

\begin{document}

\title{$1$-perfectly orientable $K_4$-minor-free and outerplanar graphs}

\author{
Bo\v{s}tjan Bre\v{s}ar$^{a,b}$\thanks{Email: \texttt{bostjan.bresar@um.si}}
\and
Tatiana Romina Hartinger$^{c,d}$\thanks{Email: \texttt{tatiana.hartinger@iam.upr.si}}
\and
Tim Kos$^{b}$\thanks{Email: \texttt{tim.kos@imfm.si}}
\and
Martin Milani\v c$^{c,d}$\thanks{Email: \texttt{martin.milanic@upr.si}}
}

\date{\today}

\maketitle

\begin{center}
$^a$ Faculty of Natural Sciences and Mathematics, University of Maribor, Slovenia

$^b$ Institute of Mathematics, Physics and Mechanics, Ljubljana, Slovenia

$^c$ University of Primorska, UP IAM, Muzejski trg 2, SI 6000 Koper, Slovenia

$^d$ University of Primorska, UP FAMNIT, Glagolja\v ska 8, SI 6000 Koper, Slovenia
\end{center}

\begin{abstract}
A graph $G$ is said to be $1$-perfectly orientable if it has an orientation such that for every vertex $v\in V(G)$, the out-neighborhood of $v$ in $D$ is a clique in $G$. In $1982$, Skrien posed the problem of characterizing the class of $1$-perfectly orientable graphs.
This graph class forms a common generalization of the classes of chordal and circular arc graphs; however, while polynomially recognizable via a reduction to $2$-SAT, no structural characterization of this intriguing class of graphs is known.
Based on a reduction of the study of $1$-perfectly orientable graphs to the biconnected case, we characterize, both in terms of forbidden induced minors and in terms of composition theorems, the classes of \hbox{$1$-perfectly} orientable $K_4$-minor-free graphs and of
\hbox{$1$-perfectly} orientable outerplanar graphs. As part of our approach, we introduce a class of graphs defined similarly as the class of $2$-trees and relate the classes of graphs under consideration to two other graph classes closed under induced minors studied in the literature: cyclically orientable graphs and graphs of separability at most~$2$.
\end{abstract}


\section{Introduction}

We consider the graphs having an orientation that is an {\it out-tournament}, that is, a digraph in which the out-neighborhood of every vertex induces an orientation of a (possibly empty) complete graph. Following the terminology of Kammer and Tholey~\cite{MR3152051}, we say that an
orientation of a graph is {\it $1$-perfect} if the out-neighborhood of every vertex induces a tournament, and that a graph is \hbox{{\it $1$-perfectly orientable}} if it has a $1$-perfect orientation. The notion of $1$-perfectly orientable graphs was introduced by Skrien in 1982~\cite{MR666799}. He referred to them under the name $\{B_2\}$-graphs and posed the problem of characterizing this graph class. It is known that $1$-perfectly orientable graphs can be recognized in polynomial time via a reduction to $2$-SAT~\cite[Theorem 5.1]{MR1244934}. A polynomial time algorithm for recognizing $1$-perfectly orientable graphs that works directly on the graph was given by Urrutia and Gavril~\cite{MR1161986}. An arc reversal argument shows that $1$-perfectly orientable graphs are exactly the graphs that admit an orientation that is an in-tournament; such orientations were called {\it fraternal orientations} in several papers~\cite{MR1161986,MR1287025,MR1292980,MR1449722,MR1246675,MR2323998, MR2548660}.

While a structural understanding of $1$-perfectly orientable graphs is still an open question, partial results are known.
Bang-Jensen et al.~\cite{MR1244934} (see also~\cite{prisner1988familien}) gave characterizations of $1$-perfectly orientable line graphs and of $1$-perfectly orientable triangle-free graphs, and showed that every graph representable as the intersection graph of connected subgraphs of unicyclic graphs is $1$-perfectly orientable. This implies that all chordal graphs and all circular arc graphs are $1$-perfectly orientable, as observed already in~\cite{MR1161986} and in~\cite{MR666799}, respectively. Bang-Jensen et al.~also showed that every graph having a unique induced cycle of order at least $4$ is $1$-perfectly orientable. The subclass of $1$-perfectly orientable graphs consisting of graphs that admit an orientation that is both an in-tournament and an out-tournament was characterized in~\cite{MR666799} (see also~\cite{MR1081957}) as precisely the class of proper circular arc graphs. A characterization of $1$-perfectly orientable graphs in terms of edge clique covers and characterizations of $1$-perfectly orientable co-bipartite graphs and cographs were obtained recently by  Hartinger and Milani\v c~\cite{2014arXiv1411.6663R}.

\medskip
\noindent{\bf Overview of our results and approach.}
In this paper we prove various structural properties of $1$-perfectly orientable graphs and apply them to derive our main results:
characterizations of $1$-perfectly orientable graphs within $K_4$-minor-free graphs and outerplanar graphs, respectively.
In both cases (Theorems~\ref{thm:k4-minor-free} and~\ref{thm:outerplanar-characterization}), characterizations include a complete description of the set of forbidden induced minors as well as a composition result.
To obtain these results, we split the study of the structure of connected $1$-perfectly orientable graphs into two natural cases -- biconnected graphs and graphs with cut vertices. In the latter case, it turns out that all but possibly one block in a graph need to have the additional property that the corresponding induced subgraph has a $1$-perfect orientation with a sink, that is, a vertex with empty out-neighborhood. These two cases as a rule yield two different structural results for $1$-perfectly orientable graphs within a given class of graphs, one for the biconnected case and the other for the general case. The general case is then reduced to the study of biconnected $1$-perfectly orientable rooted graphs, where the root plays the role of the (unique) sink in some $1$-perfect orientation of the graph. Two classes come forth in the investigation of $1$-perfectly orientable $K_4$-minor-free graphs; these are the well-known class of $2$-trees and the newly introduced class of non-chordal graphs, which are obtainable by a similar inductive construction as $2$-trees and which we call {\em hollowed 2-trees}.

\medskip
The known and new results on the relationships between the graph classes studied in this paper are summarized in the Hasse diagram on Fig.~\ref{fig:Hasse}.

\begin{figure}[h!]
\begin{center}
\includegraphics[width=\textwidth]{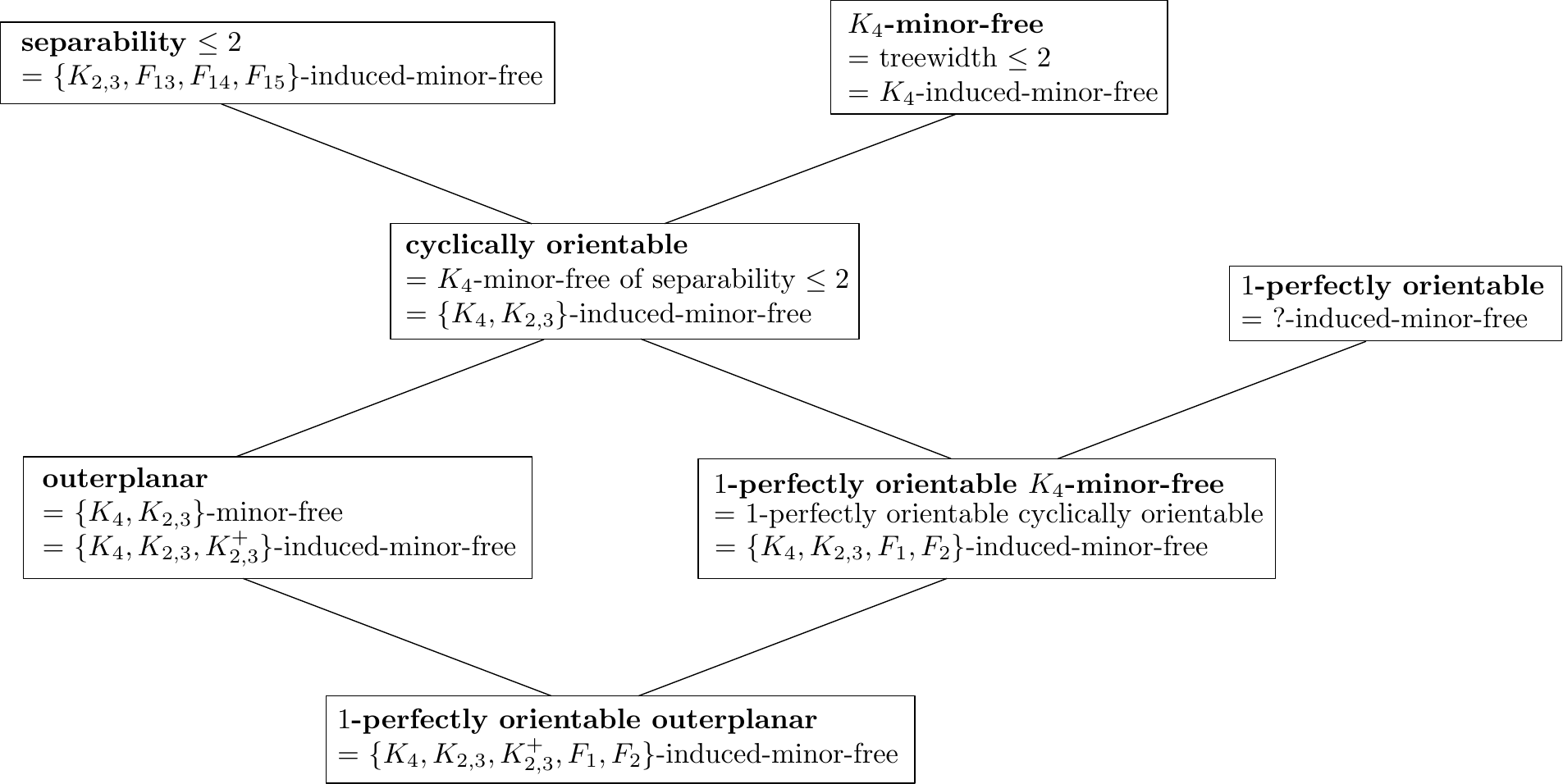}
\caption{Hasse diagram of inclusion relations between induced-minor-closed graph classes considered in this paper, summarizing known results and results obtained in this paper.}
\label{fig:Hasse}
\end{center}
\end{figure}

\medskip

\noindent{\bf Organization of the paper.} In  Section~\ref{sec:prelim} we fix the notation and state some preliminary results for later use; in particular, a result from \cite{2014arXiv1411.6663R} gives a long list of forbidden induced minors for the class of $1$-perfectly orientable graphs. In this section we also prove some preparatory results for later use; for instance, cyclically orientable graphs are characterized as exactly the $\{K_4,K_{2,3}\}$-induced-minor-free graphs and it is proved that every $1$-perfect orientation of a connected graph contains at most one sink.
In Section~\ref{sec:biconnected} the possible $1$-perfect orientations of graphs with cut-vertices are analyzed. As a by-product, connected $1$-perfectly orientable block-cactus graphs are characterized as exactly those block-cactus graphs in which at most one block is not complete.
In Section~\ref{sec:hollowed} we introduce the class of hollowed $2$-trees and present several properties of $2$-trees and of hollowed $2$-trees needed for the investigation of $1$-perfectly orientable graphs within the class of $K_4$-minor-free graphs. Finally, in Sections~\ref{sec:K4minor} and~\ref{sec:outerplanar} we consider $K_4$-minor-free graphs and outerplanar graphs, respectively, and prove several characterizations of $1$-perfectly orientable graphs within those graph classes. We conclude in Section~\ref{sec:conclusion} with an observation that $1$-perfectly orientable planar graphs are of bounded treewidth and pose a related open problem, the answer to which could lead to further insights on the structure of $1$-perfectly orientable graphs.


\section{Preliminaries}
\label{sec:prelim}
In this section, we review the basic notation and definitions, recall the basic properties of some graph classes relevant to our study, and prove some preliminary results.

We use standard graph theory notation. All graphs considered in this paper will be finite and simple but can be either undirected or directed; accordingly, we will refer to them as graphs or as digraphs, respectively. Given a graph $G$ and a subset $S$ of its vertices, the {\it subgraph of $G$ induced by $S$} is the graph denoted by $G[S]$ and defined as $(S,\{\{u,v\}: u,v\in S$ and $\{u,v\}\in E(G)\})$. Induced subgraphs of directed graphs are defined analogously. A subset of vertices in a graph is called a {\it clique} if it induces a complete graph. The {\it distance} between two vertices $u$ and $v$ in an undirected connected graph $G$ is denoted by $d_G(u,v)$ and defined as the minimum length (that is, number of edges) of a $u$,$v$-path in $G$.
 A {\it cut vertex} in a connected graph $G$ is a vertex $v$ such that the graph $G-v$ is disconnected.
A graph $G$ is {\it biconnected} if it is connected and has no cut vertices, and {\it $2$-connected}
if it is biconnected and has at least $3$ vertices.  A {\it block} of a graph $G$ is a maximal biconnected subgraph of $G$.
Every connected graph decomposes into a tree of blocks called the {\it block tree} of the graph.
The vertex set of the block tree $T$ of $G$ is the set ${\cal B}\cup C$ where ${\cal B}$ is the set of blocks of $G$ and $C$ is the set of cut vertices of $G$;
a block $B\in {\cal B}$ and a cut vertex $v\in C$ are connected by en edge in $T$ if and only if $v\in V(B)$.
Blocks of $G$ that are leaves of $T$ are called {\it end blocks} of $G$. Every leaf of the block tree $T$ is a block of $G$, thus every graph with a cut vertex
has at least two end blocks.
A {\it sink} in a directed graph is a vertex of out-degree zero.
A directed graph is said to be {\it sink-free} if it has no sinks.

\smallskip

{\bf Minors and induced minors.} A graph $H$ is said to be a {\it minor} of a graph $G$ if $H$ can be obtained from $G$ by a series of vertex deletions, edge deletions, and edge contractions, or, equivalently, if there exists an {\it minor model of $H$ in $G$}, that is,
a collection $\{S_v: v\in V(H)\}$ of pairwise disjoint subsets of $V(G)$ each inducing a connected subgraph
such that for every two adjacent vertices $u$ and $v$ of $H$, we have $\{x,y\} \in E(G)$ for some $x\in S_u$ and $y\in S_v$.
A graph $H$ is said to be an {\it induced minor} of $G$ if $H$ can be obtained from $G$ by a series of vertex deletions and edge contractions,
or, equivalently, if there exists an {\it induced minor model of $H$ in $G$}, that is,
a collection $\{S_v: v\in V(H)\}$ of pairwise disjoint subsets of $V(G)$ each inducing a connected subgraph
such that for every two distinct vertices $u$ and $v$ of $H$, we have $\{u,v\}\in E(H)$ if and only if $\{x,y\} \in E(G)$ for some $x\in S_u$ and $y\in S_v$.

Given a set ${\cal F}$ of graphs, a graph $G$ is said to be {\it ${\cal F}$-minor-free}
if no minor of $G$ is isomorphic to a member of $\cal F$.
Every minor-closed class ${\cal G}$ of graphs can be uniquely characterized in terms of {\it forbidden minors}.
That is, there exists a unique set $\cal F$ of graphs such that: (i) a graph $G$ is in ${\cal G}$ if and only if
$G$ is $\cal F$-minor-free, and (ii) every proper minor of every graph in $\cal F$ is in ${\cal G}$.
The notions of {\it ${\cal F}$-induced-minor-free} graphs and of {\it forbidden induced minors} are defined
analogously, with respect to the induced minor relation. For minor-closed graph classes, the sets of forbidden minors are always finite, while in the case of
induced-minor-closed graph classes, the sets of forbidden induced minors can be either finite or infinite.
Given two graphs $G$ and $H$, we say that $G$ is {\it $H$-free} \hbox{({\it $H$-minor-free}}, resp., {\it $H$-induced-minor-free})
if no induced subgraph of $G$ (no minor of $G$, resp., no induced minor of $G$) is isomorphic to $H$.

Note that for every graph $H$ an induced minor model of $H$ in $G$ is also a minor model of it, and if $H$ is a complete graph, then the converse holds as well. In particular, for $H = K_4$, this implies that a graph $G$ is $K_4$-minor-free if and only if it is $K_4$-induced-minor-free.
Moreover, if $H$ is a graph of maximum degree at most three and $G$ is any graph, then $H$ is isomorphic to a minor of $G$
if and only if $G$ contains a subgraph isomorphic to a subdivision of $H$. In particular, for $H = K_4$ we obtain that a graph $G$ is $K_4$-minor-free if and only if it does not contain a subgraph isomorphic to a subdivision of $K_4$.

\smallskip

{\bf Chordal graphs.} A graph $G$ is said to  be \emph{chordal} if it is $C_k$-free for all $k\ge 4$.
A vertex in a graph $G$ is {\it simplicial} if its neighborhood forms a clique. A {\it perfect elimination ordering} in a graph is a linear ordering of the vertices of the graph such that, for each vertex $v$, the neighbors of $v$ that occur after $v$ in the order form a clique.
Fulkerson and Gross showed that graph is chordal if and only if it has a perfect elimination ordering~\cite{MR0186421}; equivalently, if it can be reduced to the one-vertex graph by a sequence of simplicial vertex removals.
Note that the class of chordal graphs is closed both under vertex deletions and edge contractions, hence it is also closed under induced minors.
Consequently, a graph $G$ is chordal if and only if it is $C_4$-induced-minor-free.

For further background on graph theory we refer the reader to~\cite{opac-b1096791,MR2744811}.

\subsection{Preliminaries on $1$-perfectly orientable graphs, graphs of separability at most $2$, cyclically orientable graphs, and outerplanar graphs}

The class of $1$-perfectly orientable graphs is closed under induced minors~\cite{2014arXiv1411.6663R}.
The characterization of the class of $1$-perfectly orientable graphs in terms of forbidden induced minors is not known;
a~partial answer is given in the following theorem.

\begin{sloppypar}
\begin{theorem}[Hartinger and Milani\v c~\cite{2014arXiv1411.6663R}]\label{prop:1PO}
Let ${\cal F} = \{F_1,F_2,F_5,\ldots,F_{12}\}\cup {\cal F}_3\cup {\cal F}_4$, where: \begin{itemize}
   \item graphs $F_1$, $F_2$ are depicted in Fig.~\ref{fig:1}, and
  \item ${\cal F}_3 = \{\overline{C_{2k}}\mid k\ge 3\}$, the set of complements of even cycles of length at least $6$,
  \item ${\cal F}_4 = \{\overline{K_2+C_{2k+1}}\mid k\ge 1\}$, the set of complements of the graphs obtained as the disjoint union of $K_2$ and an
      odd cycle,
  \item for $i \in \{5,\ldots, 12\}$, graph $F_i$ is the complement of the graph $G_{i-4}$, depicted in Fig.~\ref{fig:1}.
\end{itemize} Then, every $1$-perfectly orientable graph is ${\cal F}$-induced-minor-free.
\end{theorem}
\end{sloppypar}

 \begin{figure}[h!]
  \centering
\includegraphics[width=0.9\textwidth]{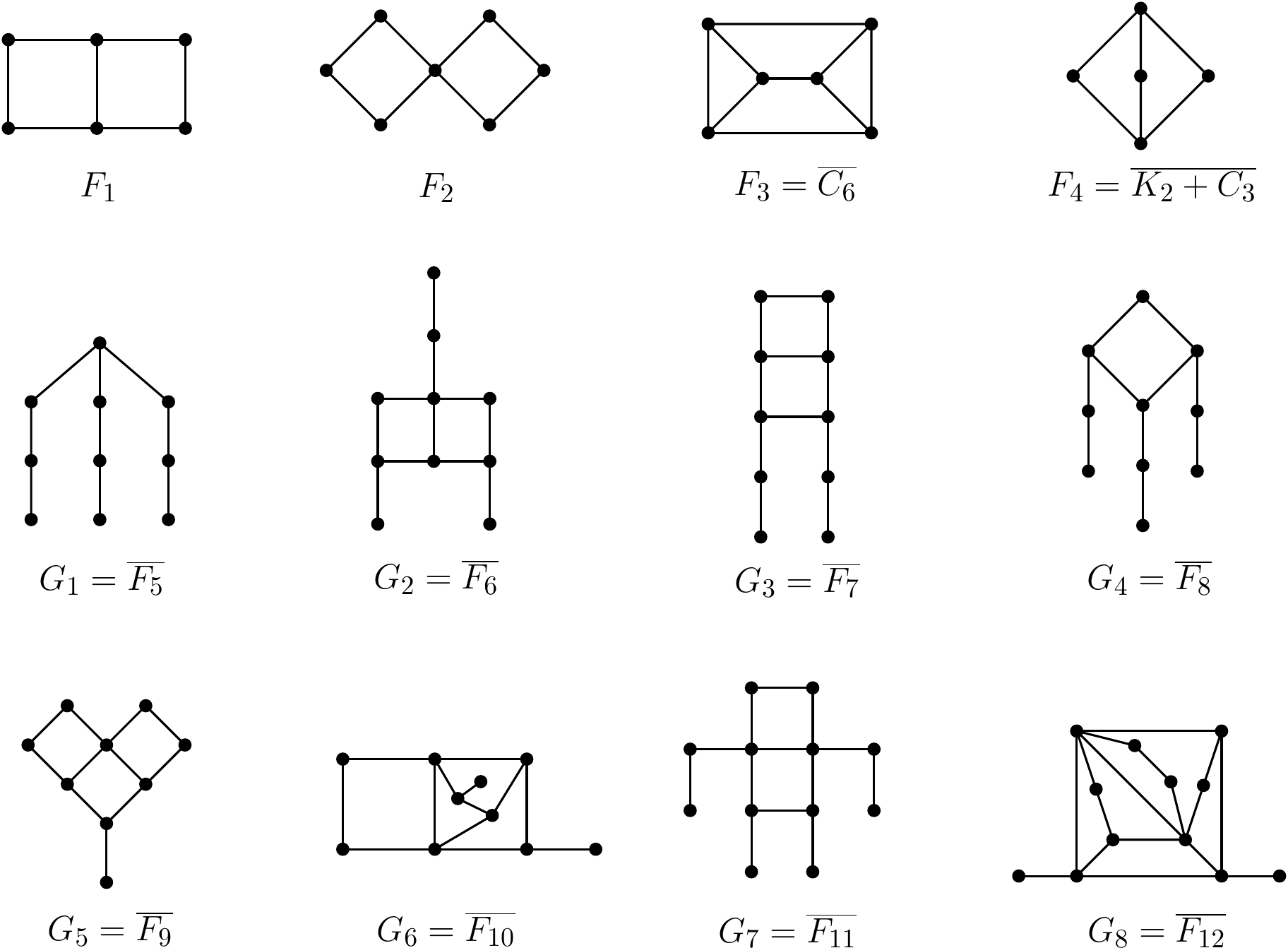}
\caption{Four non-$1$-perfectly orientable~graphs and $8$ complements of non-$1$-perfectly orientable graphs. Graphs $F_3$ and $F_4$ are the smallest members of families ${\cal F}_3$ and ${\cal F}_4$, respectively.}\label{fig:1}
\end{figure}

An orientation of a cycle is said to be {\it cyclic} if it has no sink.
The following simple observation about $1$-perfectly orientable graphs is also from~\cite{2014arXiv1411.6663R}.

\begin{lemma}\label{lem:cycles}
In every $1$-perfect orientation of a graph $G$, every induced cycle of $G$ of length at least $4$ is oriented cyclically.
\end{lemma}

For a positive integer $k$, {\it graphs of separability at most $k$} were defined by Cicalese and Milani\v c in~\cite{MR2901082}
as the graphs in which every two non-adjacent vertices are separated by a set of at most $k$ other vertices.
Several characterizations of graphs of separability at most $2$ were given in~\cite{MR2901082}. In the next theorem we summarize those relevant to this paper (Theorems 1 and 9 in~\cite{MR2901082}). We say that a graph $G$ is obtained from two graphs $G_1$ and $G_2$ by {\em pasting along a $k$-clique}, and denote this by $G = G_1\oplus_k G_2$,
if for some $r\le k$ there exist two $r$-cliques $K^{(1)} = \{x_1,\ldots, x_r\}\subseteq V(G_1)$  and $K^{(2)} = \{y_1,\ldots, y_r\}\subseteq V(G_2)$ such that $G$ is isomorphic to the graph obtained from the disjoint union of $G_1$ and $G_2$ by identifying each $x_i$ with $y_i$, for all $i=1,\ldots, r$. In particular, if $k=0$, then $G_1\oplus_kG_2$ is the disjoint union of $G_1$ and $G_2$, and if $k=1$, then the graph $G_1\oplus_kG_2$ has a cut vertex.

\begin{theorem}[Cicalese-Milani\v c~\cite{MR2901082}]\label{thm:separability-2}
For every graph $G$, the following statements are equivalent.
\begin{enumerate}
  \item $G$ is of separability at most $2$.
  \item $G$ is $\{K_{2,3}, F_{13},F_{14},F_{15}\}$-induced-minor-free, where $K_{2,3}, F_{13},F_{14},F_{15}$ are the four graphs depicted in Fig.~\ref{fig:forb}.
  \item $G$ can be constructed from complete graphs and cycles by an iterative application of pasting along $2$-cliques.
\end{enumerate}
\end{theorem}

\begin{figure}[h!]
\begin{center}
\includegraphics[height=24mm]{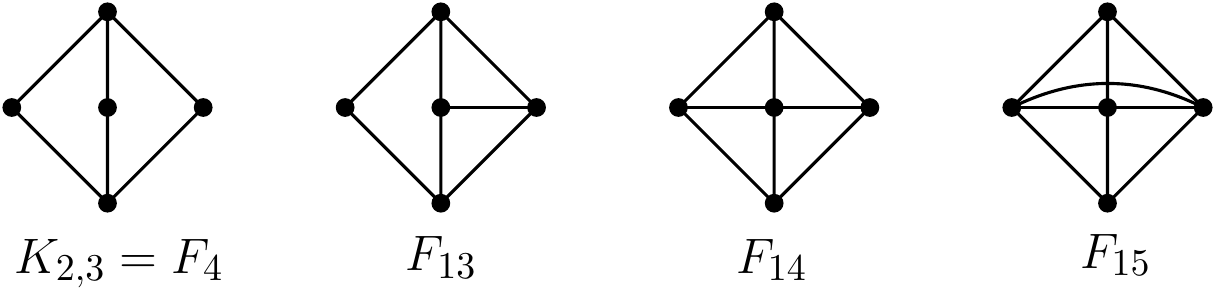}
\caption{Forbidden induced minors for the class of graphs of separability at most 2}\label{fig:forb}
\end{center}
\end{figure}

A graph $G$ is said to be {\it cyclically orientable} if it admits an orientation in which every chordless cycle is oriented cyclically. Motivated by applications of cyclically orientable graphs to cluster algebras, this family of graphs was introduced by Barot et al.~\cite{MR2241966} and studied further by Gurvich~\cite{MR2370526} and Zou~\cite{MR2643278}. The following theorem combines~\cite[Theorems 1 and 4]{MR2370526} and~\cite[Theorem 1]{Speyer}.

\begin{theorem}[Gurvich~\cite{MR2370526}; Speyer~\cite{Speyer}]\label{thm:CO}
For every graph $G$, the following statements are equivalent:
\begin{enumerate}
  \item $G$ is cyclically orientable.
  \item $G$ can be built from copies of $K_1$, $K_2$, and cycles
  by an iterative application of pasting along $2$-cliques.
  \item $G$ is a $K_4$-free graph of separability at most $2$.
\end{enumerate}
\end{theorem}

As a consequence, we obtain the following.

\begin{corollary}\label{cor:CO}
If $G$ is biconnected and cyclically orientable, then $G$ can be build from cycles by an iterative application of pasting along edges.
\end{corollary}

The following is~\cite[Lemma 2]{MR2370526}.

\begin{lemma}[Gurvich~\cite{MR2370526}]\label{lem:CO-K4}
Cyclically orientable graphs contain no subgraphs isomorphic to a subdivision of $K_4$.
\end{lemma}

A graph $G$ is {\it outerplanar} if it can be drawn in
the plane without edge crossings and with all vertices incident with the outer face. Outerplanar graphs are exactly the $\{K_4,K_{2,3}\}$-minor-free graphs~\cite{Chartrand1967}. The following characterization in terms of forbidden induced minors is an immediate consequence of the characterization in terms of forbidden minors. We denote by $K_{2,3}^+$ the graph obtained from $K_{2,3}$ by adding an edge between the two vertices of degree $3$.

\begin{proposition}\label{prop:outerplanar}
For every graph $G$, the following statements are equivalent:
\begin{enumerate}
  \item $G$ is outerplanar.
  \item $G$ is $\{K_4, K_{2,3}, K_{2,3}^{+}\}$-induced-minor-free.
 \end{enumerate}
\end{proposition}

\begin{proof}
Since none of the graphs $K_4$, $K_{2,3}$, $K_{2,3}^{+}$ is outerplanar (as it has either a $K_4$ or a $K_{2,3}$ as a minor), the class of outerplanar graphs is contained in the class of  $\{K_4, K_{2,3}, K_{2,3}^{+}\}$-induced-minor-free graphs.
Conversely, we will show that every $\{K_{2,3},K_{2,3}^{+},K_4\}$-induced-minor-free graph $G$ is $\{K_4,K_{2,3}\}$-minor-free (and hence outerplanar). Indeed, suppose that $G$ contains $H\in \{K_4,K_{2,3}\}$ as a minor. If $H = K_4$ then $G$ contains $K_4$ as induced minor, which is impossible. So $H = K_{2,3}$.
Consider a minor model ${\cal S} = \{S_v:v\in V(K_{2,3})\}$ of $K_{2,3}$ in $G$ and let $x$ and $y$ be the two vertices of degree $3$ in $K_{2,3}$. To avoid that ${\cal S}$ is an induced minor model of $K_{2,3}$ in $G$, we infer that $G$ has an edge $\{x,y\}$ for some $x\in S_u$ and $y\in S_v$. This implies that either $K_{2,3}^{+}$ or $K_4$ is an induced minor of $G$, contrary to the assumption.
\end{proof}

\subsection{Preparatory results}

The above results imply the following characterization of cyclically orientable graphs in terms of forbidden induced minors.

\begin{theorem}\label{thm:K4K23-i-m-free}
For every graph $G$, the following statements are equivalent:
\begin{enumerate}
  \item $G$ is cyclically orientable.
  \item $G$ is $\{K_4,K_{2,3}\}$-induced-minor-free.
 \end{enumerate}
\end{theorem}

\begin{proof}
We first argue that the class of cyclically orientable graphs is closed under induced minors.
It is clearly closed under vertex deletions. To see that it is also closed under edge contractions, recall that
by Theorem~\ref{thm:CO}
$G$ is cyclically orientable if and only if $G$ is a $K_4$-free graph of separability at most $2$.
Since the class of graphs of separability at most $2$ is closed under induced minors (cf.~Theorem~\ref{thm:separability-2}),
contracting an edge of a cyclically orientable graph $G$ results in a graph $G'$ of separability at most $2$.
By Lemma~\ref{lem:CO-K4}, $G$ does not contain any subdivision of $K_4$ (as a subgraph), which is equivalent to
the fact that $G$ does not contain $K_4$ as a minor. Since contracting an edge cannot produce a $K_4$ minor,
graph $G'$ has no $K_4$ minor, in particular, it is $K_4$-free. Thus, $G'$ is cyclically orientable by Theorem~\ref{thm:CO}.

Since the class of cyclically orientable graphs is closed under induced minors and
the graphs $K_4$ and $K_{2,3}$ are not cyclically orientable, the implication $1\Rightarrow 2$ follows.
Suppose now that $G$ is a $\{K_4,K_{2,3}\}$-induced-minor-free graph.
Since each of the graphs in the set $\{F_{13}, F_{14}, F_{15}\}$  (cf.~Fig.~\ref{fig:forb})
can be contracted to a $K_4$, the class of $\{K_4,K_{2,3}\}$-induced-minor-free graphs is a subclass of the class of
$\{K_{2,3},F_{13}, F_{14}, F_{15}\}$-induced-minor-free graphs.
It follows from Theorem~\ref{thm:separability-2} that
every $\{K_4,K_{2,3}\}$-induced-minor-free graph is a $K_4$-free graph of separability at most $2$.
The implication $2\Rightarrow 1$ now follows from Theorem~\ref{thm:CO}.
\end{proof}

We conclude this section with two lemmas, one regarding sinks in $1$-perfect orientations of connected graphs
and one characterizing $1$-perfect orientations of trees.

\begin{lemma}\label{lem:sink}
Every $1$-perfect orientation of a connected graph has at most one sink.
\end{lemma}

\begin{proof}
Let $D$ be a $1$-perfect orientation of a connected graph $G$ with two sinks $x$ and $y$.
Suppose for a contradiction that $x\neq y$.
Let $P = (x = v_0,v_1,\ldots, v_k = y)$ be a shortest $x$,$y$-path in $G$.
Since $v_0 = x$ is a sink, the edge $\{v_0,v_1\}$ is oriented as $v_1\to v_0$ in $D$.
This implies that there is a unique maximum index $j\in \{1,\ldots, k\}$ such that $(v_{j},v_{j-1})$ is an arc of $D$.
Since $y$ is a sink, the edge $\{v_{k-1},v_k\}$ is oriented as $v_{k-1}\to v_k$ in $D$, which implies $j<k$.
The definition of $j$ implies that $(v_{j+1},v_{j})$ is not an arc of $D$. Hence
$(v_{j},v_{j+1})$ is an arc of $D$. Since the out-neighborhood of $v_j$ in $D$ is a clique in $G$,
vertices $v_{j-1}$ and $v_{j+1}$ are adjacent, contradicting the minimality of $P$.
\end{proof}

An {\it in-tree} is a directed graph $D$ that has a vertex $r$ called the {\it root} such that for every vertex $v\in V(D)$,
there is exactly one directed path from $v$ to $r$. Equivalently, an in-tree is a directed rooted tree in which
all arcs point towards the root (that is, for every edge $\{x,y\}$ of the underlying undirected tree $T$, we have that $(x,y)$ is an arc of $D$
if and only if $d_T(x,r) < d_T(y,r)$). It is easy to see in every in-tree $D$, every vertex $v\in V(D)$ satisfies $d^+_D(v) = 1$, except for the
root $r$, which is a sink.

\begin{lemma}\label{lem:tree}
Let $T$ be a tree and let $D$ be an orientation of $T$.
Then, $D$ is $1$-perfect if and only if $D$ is an in-tree.
Moreover, for every vertex $r\in V(T)$ there exists a $1$-perfect orientation $D$
of $T$ such that $r$ is the root of the in-tree $D$.
\end{lemma}

\begin{proof}
If $D$ is an in-tree then the fact that $D$ is a $1$-perfect orientation of $T$ follows from the fact that
every vertex $v\in V(T)$ satisfies $d^+_D(v) \le 1$.

Suppose now that $D$ is a $1$-perfect orientation of $T$. Then $d^+_D(v) \le 1$ for all vertices $v\in V(T)$.
By Lemma~\ref{lem:sink}, $D$ has at most one sink.
If $D$ does not have any sink, then for every vertex $v$ of $T$, we have $d^{+}_{D}(v)= 1$,
which implies that the total number of arcs in $D$ equals $|V(D)| = |V(T)|$, contrary to the fact that $T$ is acyclic.
It follows that $D$ has a unique sink, say $r$. We claim that $D$ is an in-tree with root $r$, that is,
for every vertex $v\in V(D)$ there is exactly one directed path from $v$ to $r$.
Since $D$ is an orientation of a tree, for every vertex $v\in V(T)$
there is at most one directed path from $v$ to $r$. Clearly, for $v = r$ there is a unique $v$,$r$-directed path.
Since every vertex $v$ that is not the root has a unique out-neighbor and $D$ has no directed cycles, any maximal path from $v$
ends in the root.

The last statement of the lemma is immediate since, given a vertex $r\in V(T)$,
orienting all edges of $T$ towards $r$ results in an in-tree with root $r$.
\end{proof}


\section{Reduction to the biconnected case}
\label{sec:biconnected}
Since a graph is $1$-perfectly orientable~if and only if each component of $G$ is $1$-perfectly orientable, in the study of $1$-perfectly orientable~graphs we may restrict our attention to connected graphs. In this section, we analyze to what extent the study of $1$-perfectly orientable~graphs can be reduced to the
biconnected case. It turns out that biconnectivity comes at a price: the study of slightly more general structures is required, namely of pairs $(G,v)$ where $G$ is a biconnected $1$-perfectly orientable~graph having a $1$-perfect orientation $D$ such that $v$ is a sink in $D$.

\begin{definition}
A {\it rooted graph} is a pair $(G,v)$, denoted also by $G^v$, such that $G$ is a graph and $v\in V(G)$.
A rooted graph $G^v$ is said to be {\it connected} (resp., biconnected) if $G$ is connected (resp., biconnected), and
{\it $1$-perfectly orientable} if $G$ has a $1$-perfect orientation in which $v$ is a sink.
\end{definition}

We summarize the reduction to biconnected rooted graphs in the following theorem. Given a tree $T$ and a vertex $r\in V(T)$, the $1$-perfect orientation of $T$ in which $r$ is the unique sink (cf.~Lemma~\ref{lem:tree}) will be referred to as the {\it $r$-rooted orientation of $T$}.

\begin{sloppypar}
\begin{theorem}\label{thm:biconnected-reduction}
Let $G$ be a connected graph with a cut vertex, let ${\cal B}$ and $C$ be the sets of blocks and
cut vertices of $G$, respectively, and let $T$ be the block tree of $G$. Then, $G$ is $1$-perfectly orientable~if and only if one of the following conditions holds:
\begin{enumerate}
  \item There exists a block $B_r$ of $G$ such that $B_r$ is $1$-perfectly orientable~and
  for every arc \hbox{$(B,v)\in {\cal B}\times C$} of the $B_r$-rooted orientation of $T$, the rooted graph $B^v$ is $1$-perfectly orientable.
  \item There exists a cut vertex $v_r$ of $G$ such that
  for every arc $(B,v)\in {\cal B}\times C$ of the $v_r$-rooted orientation of $T$, the rooted graph $B^v$ is $1$-perfectly orientable.
  \end{enumerate}
\end{theorem}
\end{sloppypar}

\begin{proof}
{\it Necessity.}
Suppose first that $G$ is $1$-perfectly orientable, and let $D$ be a $1$-perfect orientation of $G$.
Consider the orientation $T_D$ of the block tree $T$ defined by orienting any edge $\{v,B\}$ of $T$
(with $v\in C$ and $B\in {\cal B}$) as $B\to v$ if and only if
$v$ is a sink in the subgraph of $D$ induced by $V(B)$.

We claim that for every node $x$ of the block tree $T$, we have $d^{+}_{T_D}(x)\le 1$.
If $x=B$ is a block of $G$, then the inequality $d^{+}_{T_D}(B)\le 1$ follows from
Lemma~\ref{lem:sink}. So let $x = v$ be a cut vertex of $G$ and suppose
for a contradiction that $d^+_{D_T}(v)\ge 2$. Then there exist two blocks $B$ and $B'$ of $G$ containing $v$
such that $v$ is not a sink in the subgraph of $D$ induced by $X$ for any $X\in \{V(B),V(B')\}$.
It follows that the out-neighborhood of $v$ in $D$ contains a vertex from $V(B)\setminus\{v\}$ and
a vertex from $V(B')\setminus\{v\}$. As these two vertices are not adjacent in $G$, this contradicts the fact that $D$ is $1$-perfect.

Since every node of $T_D$ is of out-degree at most $1$, $T_D$ is a $1$-perfect orientation of $T$. Lemma~\ref{lem:tree} implies that orientation $T_D$ is an in-tree. As there are two types of nodes in $T$, the blocks of $G$ and the cut vertices of $G$, the unique sink of $T_D$ can be either a block of $G$ or a cut vertex.
Suppose first that the unique sink of $T_D$ is a block of $G$, say $B_r$. Then $T_D$ is the $B_r$-rooted orientation of $T$.
Since the subgraph of $D$ induced by $V(B_r)$ is a $1$-perfect orientation of $B_r$,
it follows that $B_r$ is $1$-perfectly orientable. Moreover, for every arc $(B,v)\in {\cal B}\times C$ of $T_D$,
the subgraph of $D$ induced by $V(B)$ is $1$-perfect orientation of $B$ in which $v\in V(B)$ is a sink,
which implies that $B^v$ is $1$-perfectly orientable. Thus, condition 1 holds in this case.
A similar argument shows that condition 2 holds if the unique sink of $T_D$ is a cut vertex of $G$.

\medskip
{\it Sufficiency.} We now show that each of the two conditions is sufficient for $G$ to be $1$-perfectly orientable.

First, suppose that condition 1 holds, that is, there exists a block $B_r$ of $G$ such that $B_r$ is $1$-perfectly orientable~and
for every arc $(B,v)\in {\cal B}\times C$ of the $B_r$-rooted orientation of $T$, the rooted graph $B^v$ is $1$-perfectly orientable.
Fix a $1$-perfect orientation $D_{B_r}$ of $B_r$ and, for every arc $(B,v)\in {\cal B}\times C$
of the $B_r$-rooted orientation of $T$,
fix a $1$-perfect orientation $D_B$ of $B$ in which $v$ is a sink.
Note that since every block $B\neq B_r$ is of out-degree $1$ in the $B_r$-rooted orientation of $T$,
each of the blocks of $G$ is oriented by exactly one of the above $|{\cal B}|$ orientations. Since
each edge of $G$ lies in a unique block of $G$, combining the above
orientations defines a unique orientation of $G$, say~$D$.
We claim that $D$ is a $1$-perfect orientation of $G$.
Every vertex $v\in V(G)$ that is not a cut vertex belongs to a unique block, say $B$, and therefore
$N^+_D(v) = N^+_{D_{B}}(v)$. Since $D_{B}$ is a $1$-perfect orientation of $B$, the set
$N^+_{D_{B}}(v)$ is a clique in $B$, and hence also a clique in $G$.
If $v\in V(G)$ is a cut vertex, then there is a unique block $B$ of $G$ such that $(v,B)$ is an arc of the $B_r$-rooted orientation of $T$,
which means that for every block $B'$ containing $v$ other than $B$, the vertex $v$ is a sink in $D_{B'}$.
Again we obtain that $N^+_D(v) = N^+_{D_{B}}(v)$, hence this set is a clique in $G$.
This shows that $D$ is a $1$-perfect orientation of $G$, showing that $G$ is $1$-perfectly orientable.

The proof in the case when condition 2 holds is very similar.
Suppose that there is a cut vertex $v_r$ of $G$ such that
for every arc $(B,v)\in {\cal B}\times C$ of the $v_r$-rooted orientation of $T$, the rooted graph $B^v$ is $1$-perfectly orientable.
For every such arc $(B,v)$, fix a $1$-perfect orientation $D_B$ of $B$ in which $v$ is a sink.
In this case, every block $B$ of $G$ is of out-degree $1$ in the $v_r$-rooted orientation of $T$ and
combining the above $|{\cal B}|$ orientations defines a unique orientation of $G$, say $D$.
Arguments analogous to those in  the above paragraph show that $D$ is a $1$-perfect orientation of $G$,
hence $G$ is $1$-perfectly orientable in this case too. This completes the proof.
\end{proof}

As an immediate consequence of the previous theorem, we obtain the following side result: a characterization of $1$-perfectly orientable block-cactus graphs. A \emph{block-cactus graphs} is a graph such that all its blocks are either cycles or complete graphs.
A {\it rooted extension of a graph $G$} is a rooted graph $G^v$ for any $v\in V(G)$.

\begin{corollary}
A connected block-cactus graph $G$ is $1$-perfectly orientable if and only if at most one block of $G$ is a cycle of length at least four.
\end{corollary}

\begin{proof}
Let $G$ be a connected block-cactus graph. Then each block of $G$ is a cycle or a complete graph, hence $1$-perfectly orientable.
If $G$ is $2$-connected, then $G$ has only one block and the condition from the corollary is trivially satisfied.
Suppose now that $G$ has a cut-vertex. Graph $G$ has blocks of two types: blocks that are complete -- for which every rooted extension is $1$-perfectly orientable -- and blocks that are not complete -- which are cycles of length at least four, and for which, by Lemma~\ref{lem:cycles}, no rooted extension is \hbox{$1$-perfectly orientable}. The two conditions from Theorem~\ref{thm:biconnected-reduction} are now easily seen to be equivalent to the
following two conditions, respectively: (1) there exists a block $B$ of $G$ such that all blocks of $G$ other than $B$ are complete, and (2)
all blocks of $G$ are complete. Clearly, at least one of these two conditions holds if and only if at most one block of $G$ is a cycle of length at least four.
\end{proof}

We also prove a lemma on chordal graphs for later use.

\begin{lemma}\label{lem:chordal}
Every rooted extension of a chordal graph is $1$-perfectly orientable.
\end{lemma}

\begin{proof}
Let $G$ be a chordal graph and $v\in V(G)$. Since $G$ is chordal, it has a perfect elimination ordering, that is, a linear ordering $\sigma= (v_1,\ldots, v_n)$ of the vertices of $G$ such that for all \hbox{$i\in \{1,\ldots, n\}$}, vertex $v_i$ is a simplicial vertex in the subgraph of $G$ induced by $\{v_1,\ldots, v_n\}$. Moreover,
the perfect elimination orderings of $G$ are exactly the sequences of the form $(\sigma',v_n)$ where $v_n$ is a simplicial vertex of $G$ and $\sigma'$ is a perfect elimination ordering of $G-v_n$.

We claim that $G$ has a perfect elimination ordering $\sigma= (v_1,\ldots, v_n)$ such that $v = v_1$. As observed already by Dirac~\cite{MR0130190},
every minimal separator in a chordal graph is a clique, which implies that every chordal graph is either complete or has a pair of non-adjacent simplicial vertices.
It follows that every chordal graph with at least two vertices has a pair of perfect elimination orderings $\sigma = (u_1,\ldots, u_n)$ and
$\sigma' = (u_1',\ldots, u_n')$ such that $u_n\neq u_n'$. In particular, one can construct
a perfect elimination ordering $(v_1,\ldots, v_n)$ of $G$ by iteratively deleting simplicial vertices
(and at the end reversing the order of deleted vertices) so that vertex $v$ is deleted only at the very end, that is, so that $v = v_1$, as claimed.

Let $\sigma= (v_1,\ldots, v_n)$ be a perfect elimination ordering of $G$ such that $v = v_1$.
Orienting the edges of $G$ as $v_i\to v_j$ if and only if $i>j$ result in a $1$-perfect orientation of $G$ in which $v$ is a sink, showing that
$G^v$ is $1$-perfectly orientable.
\end{proof}


\section{Hollowed $2$-trees}
\label{sec:hollowed}

It is well known that trees can be constructed recursively as follows:
(i) $K_1$ is a tree, (ii) a graph obtained from a tree by adding to it a vertex of degree $1$ is a tree, and (iii) there are no other trees.
The class of {\it $2$-trees} is defined in a similar way: (i) $K_2$ is a $2$-tree, (ii) a graph obtained from a $2$-tree by adding to it a simplicial vertex of degree $2$ is a $2$-tree, and (iii) there are no other $2$-trees. We now consider the following extension of the notion of $2$-trees.

\begin{definition}
A \emph{hollowed $2$-tree} is defined as follows: (i) any cycle of length at least four is a hollowed $2$-tree, (ii) a graph obtained from a hollowed $2$-tree by adding to it a simplicial vertex of degree $2$ is a hollowed $2$-tree,
and (iii) there are no other hollowed $2$-trees.
\end{definition}

The name of this graph class relates to the fact that a {\it hole} in a graph $G$ often refers to an induced cycle of length at least four in $G$.
Every hollowed $2$-tree has a unique hole (and, in particular, is not a 2-tree).

Note that all $2$-trees and all hollowed $2$-trees are biconnected. They will play an important role in our
characterization of $1$-perfectly orientable $K_4$-minor-free graphs (Theorem~\ref{thm:k4-minor-free})
and in its reduction to the biconnected case.

We first note some properties of $1$-perfect orientations of $2$-trees and of hollowed $2$-trees.

\begin{lemma}\label{lem:1-perfect-orientations-2-trees}
All $2$-trees and their rooted extensions are $1$-perfectly orientable.
Every hollowed $2$-tree is $1$-perfectly orientable, however, all its $1$-perfect orientations are sink-free. (That is, no rooted extension of a hollowed $2$-tree is $1$-perfectly orientable.)
\end{lemma}

\begin{proof}
Since $2$-trees are chordal, Lemma~\ref{lem:chordal} implies that all their rooted extensions are $1$-perfectly orientable.
In particular, every $2$-tree is $1$-perfectly orientable.

Now, let $G$ be a hollowed $2$-tree.
We prove by induction on $|V(G)|$ that $G$ is $1$-perfectly orientable, having only sink-free $1$-perfect orientations.
If $G$ is a cycle of length at least $4$, then this holds by Lemma~\ref{lem:cycles}.
Otherwise, $G$ is obtained from a hollowed $2$-tree $G'$ by adding to it a simplicial vertex, say $v$, of degree $2$.
Extending a $1$-perfect orientation of $G'$ by orienting the two edges incident with $v$ away from $v$ yields a $1$-perfect orientation of $G$,
hence $G$ is $1$-perfectly orientable. Suppose for a contradiction that $G$ has a $1$-perfect orientation $D$ with a sink $s$. If $s \neq v$, then the subgraph of $D$ induced by $V(G')$ would be a $1$-perfect orientation of $G'$ with a sink, contrary to the inductive hypothesis. Therefore $s = v$.
Let $x$ and $y$ be the two neighbors of $v$ and suppose without loss of generality that $x\to y$ in $D$.
Since $D$ is a $1$-perfect orientation of $G$, we infer that $y$ is a sink in $D'$, the subgraph of $D$ induced by
$V(G')$. However, this implies that $D'$ is a $1$-perfect orientation of $G'$ with a sink, contrary to the inductive hypothesis.
\end{proof}

In the rest of the section, we prove four inter-related lemmas:
one regarding $K_4$-minor-free biconnected graphs,
one showing that $2$-trees and hollowed $2$-trees are the only biconnected $\{K_4,K_{2,3},F_1\}$-induced-minor-free graphs,
one showing that every connected $\{K_4, K_{2,3}, F_1, F_2\}$-induced-minor-free graph has at most one hole,
and, finally, one characterizing the $\{K_{2,3},F_1,F_2\}$-induced-minor-free graphs within the class of
connected $K_4$-minor-free graphs.

\begin{lemma}\label{lem:2-trees}
Let $G$ be a biconnected $K_4$-minor-free graph with at least two vertices.
Then $G$ is chordal if and only if $G$ is a $2$-tree.
\end{lemma}

\begin{proof}
It follows immediately from the definition of $2$-trees that every $2$-tree is chordal.

Conversely, suppose that $G$ is a biconnected chordal $K_4$-minor-free graph with at least two vertices.
The fact that $G$ is a $2$-tree can be proved by induction on the number of vertices.
If $G$ has exactly $2$ vertices, then $G = K_2$ is a $2$-tree.
Suppose that $|V(G)|>2$. Since $G$ is chordal, it has a simplicial vertex, say $v$.
Since $G$ is $K_4$-free, $v$ is of degree at most $2$. Since $G$ is biconnected, $v$ is of degree at least $2$.
Therefore, $v$ is of degree exactly $2$. It is easy to see that the graph $G-v$ is a biconnected
 chordal $K_4$-minor-free graph with at least two vertices. Therefore, by the inductive hypothesis, $G-v$ is a $2$-tree.
 It follows that $G$ is also a $2$-tree.
\end{proof}

\begin{lemma}\label{lem:cyclically-orientable-2-conn-1-p-o}
Let $G$ be a biconnected $K_4$-minor-free  graph. Then, $G$ is $\{K_{2,3},F_1\}$-induced-minor-free
if and only if $G$ is either $K_1$, a $2$-tree, or a hollowed $2$-tree.
\end{lemma}

\begin{proof}
If $G$ is either $K_1$, a $2$-tree or a hollowed $2$-tree, then $G$ has at most one hole, which immediately implies that
neither $F_1$ nor $K_{2,3}$ is an induced minor of $G$.

Suppose now that $G$ is $\{K_{2,3},F_1\}$-induced-minor-free. If $G$ is chordal, then, since $G$ is $K_4$-minor-free, it follows from Lemma~\ref{lem:2-trees} that $G$ is a $2$-tree. Therefore we may assume that $G$ is non-chordal. We will show that in this case $G$ is a hollowed $2$-tree. It follows from Theorem~\ref{thm:K4K23-i-m-free} that $G$ is cyclically orientable. By Corollary~\ref{cor:CO}, $G$ can be constructed from cycles by an iterative application of pasting along an edge. Assume that we are in step $k >1$ of this construction procedure, and assume inductively that the graph $G'$ constructed right before step $k$ is a hollowed $2$-tree. Now, in step $k$ we will paste a cycle $C$ along some edge $e=xy$ of $G'$.
If $C$ is of length $3$, the graph will remain a hollowed $2$-tree after the last operation. So we can assume that $C$ is of length at least $4$.
Let $C'$ be the unique hole in $G'$. Since $G$ is biconnected, it has a pair $P$, $Q$ of vertex-disjoint paths between
$x$ and $C'$ and between $y$ and $C'$, respectively. Let $P$ and $Q$ be chosen so that their common length $|E(P)|+|E(Q)|$ is minimized.
Let $x'$ and $y'$ be the endpoints of $P$ and $Q$ on $C'$, respectively.
Note that $G'$ contains three pairwise internally vertex-disjoint $x'$, $y'$-paths (two along $C'$ and one more through $P\cup Q$);
in particular, vertices $x'$ and $y'$ they cannot be separated by a set of less than $3$ other vertices.
Since $G'$ is cyclically orientable, it is of separability at most $2$ (by Theorem~\ref{thm:CO}).
Therefore, $x'$ and $y'$ are adjacent. Let $z$ be the neighbor of $x$ on $C$ other than $y$, and
similarly, $z'$ be the neighbor of $x'$ on $C'$ other than $y'$.
Now, the sets $V(P)$, $V(Q)$, $\{z\}$, $\{z'\}$, $V(C)\setminus\{x,y,z\}$, $V(C')\setminus\{x',y',z'\}$,
form an induced minor model of $F_1$ in $G$, contrary to the fact that $G$ is $F_1$-induced-minor-free.
\end{proof}

\begin{lemma}\label{lemma:no-c4}
Let $G$ be a connected $\{K_4, K_{2,3},F_1, F_2\}$-induced-minor-free graph. Then $G$ has at most one hole.
\end{lemma}

\begin{proof}
Let $G$ be a biconnected $\{K_4, K_{2,3},F_1, F_2\}$-induced-minor-free graph. If $G$ has a hole, then $G$ is not chordal and in this case $G$
is a hollowed $2$-tree (by Lemma~\ref{lem:cyclically-orientable-2-conn-1-p-o}). Therefore, $G$ has at most one hole.

Therefore we may assume that $G$ is not biconnected. Since each block of $G$ is $\{K_4, K_{2,3},F_1, F_2\}$-induced-minor-free, each block of $G$ can contain at most one hole. Suppose that $G$ contains two holes, say $C$ and $C'$. Then $C$ and $C'$ belong to different blocks, say $B$ and $B'$, respectively. Let $P=v_1, \ldots, v_n$ be a shortest path between $C$ and $C'$. If $n=1$ then $F_2$ appears as induced minor, a contradiction.
If $n = 2$, then we consider the adjacencies between $v_1$ and $C'$. If $v_1$ has exactly one neighbor or exactly two neighbors in $C'$ which are consecutive then we get $F_2$ as induced minor, if it has exactly two neighbors in $C'$ which are not consecutive, we get either $F_2$ or $K_{2,3}$ as induced minor, and if it has $3$ or more neighbors in $C'$ we get $K_4$ as induced minor. If $n\ge 3$, then by minimality of the path we cannot have adjacencies between the vertices of the two cycles and internal vertices of the path, and thus we may contract $n-2$ edges of $P$ to reduce it to the previous case.
\end{proof}

\begin{lemma}\label{lem:K4-minor-free-with-cut-vertex}
Let $G$ be a connected $K_4$-minor-free graph with a cut vertex.
Then, $G$ is $\{K_{2,3},F_1,F_2\}$-induced-minor-free
if and only if every block of $G$ is a $2$-tree, except possibly one, which is
 a hollowed $2$-tree.
\end{lemma}

\begin{proof}
Let $G$ be a connected $K_4$-minor-free graph with a cut vertex.

Suppose first that $G$ is $\{K_{2,3},F_1,F_2\}$-induced-minor-free.
Since every block of $G$ is $\{K_{2,3},F_1\}$-induced-minor-free, Lemma~\ref{lem:cyclically-orientable-2-conn-1-p-o}
implies that every block of $G$ is either a $2$-tree or a hollowed $2$-tree.
Suppose for a contradiction that $G$ has two distinct blocks, say $B$ and $B'$, that are not $2$-trees.
Each of these two blocks is a biconnected $K_4$-minor-free graph with at least two vertices.
Therefore, by Lemma~\ref{lem:2-trees}, neither of $B$ and $B'$ is chordal. By Lemma~\ref{lemma:no-c4},
$G$ contains at most one hole, and therefore such a pair of blocks $B$ and $B'$ cannot exist.

Suppose now that every block is a $2$-tree, except possibly one, which is a hollowed $2$-tree. Then $G$ has at most one hole.
By Lemma~\ref{lem:cyclically-orientable-2-conn-1-p-o}, every block is $\{K_{2,3},F_1\}$-induced-minor-free. Since any induced minor $K_{2,3}$ or $F_1$ can only belong to a single block, $G$ is $\{K_{2,3}, F_1\}$-induced-minor-free. It remains to show that $G$ is $F_2$-induced-minor-free. Assume by contradiction that $G$ contains $F_2$ as an induced minor. Fix an induced minor model of $F_2$ in $G$, say $S_{v_1}, \ldots, S_{v_7}$, minimizing the size of the union of the $S_{v_i}$'s. Suppose that the two four-cycles of $F_2$ are induced by vertex sets $\{v_1, v_2, v_3, v_4\}$ and $\{v_4, v_5, v_6, v_7\}$. By the minimality of the model, the set $S_{v_1} \cup S_{v_2} \cup S_{v_3}$ together with a path within $S_{v_4}$ forms a hole in $G$. Similarly, the sets $S_{v_5} \cup S_{v_6} \cup S_{v_7}$ together with a path within $S_{v_4}$ form a hole in $G$. However, since these two holes are distinct,
this contradicts the fact that $G$ has at most one hole.
\end{proof}


\section{$1$-perfectly orientable $K_4$-minor-free graphs}
\label{sec:K4minor}

In this section we develop a structural characterization of $1$-perfectly orientable graphs within the class of $K_4$-minor-free graphs. Since the class of $K_4$-minor-free graphs contains the class of outerplanar graphs, this will imply a structural characterization of $1$-perfectly orientable outerplanar graphs (developed in Section~\ref{sec:outerplanar}).

We first characterize the biconnected case and then apply Theorem~\ref{thm:biconnected-reduction} to characterize the general case.

\subsection{The biconnected case}

To apply Theorem~\ref{thm:biconnected-reduction}, we need to understand both biconnected $1$-perfectly orientable $K_4$-minor-free graphs and
biconnected $1$-perfectly orientable $K_4$-minor-free rooted graphs. Both characterizations
are easy to obtain using the results of the previous section.

\begin{lemma}\label{lem:K4-minor-free-biconnected}
For a biconnected $K_4$-minor-free graph $G$, the following statements are equivalent:
\begin{enumerate}
  \item $G$ is $1$-perfectly orientable.
  \item $G$ is $\{K_{2,3},F_1\}$-induced-minor-free.
  \item $G$ is either $K_1$, a $2$-tree, or a hollowed $2$-tree.
  \item $G$ is either $K_1$, $K_2$, or can be constructed from a cycle by a sequence of additions of simplicial vertices of degree $2$.
  \item $G$ is either $K_1$, $K_2$, or has a sink-free $1$-perfect orientation.
\end{enumerate}
\end{lemma}

\begin{proof}
The implication $1\Rightarrow 2$ follows from Theorem~\ref{prop:1PO},
Lemma~\ref{lem:cyclically-orientable-2-conn-1-p-o} yields
the equivalence $2\Leftrightarrow 3$.
The equivalence between $3$ and $4$ follows immediately from the definitions of $2$-trees and hollowed $2$-trees.
The implication $5\Rightarrow 1$ is clear.

To complete the proof, we show the implication $3\Rightarrow 5$.
Suppose that $G$ is either $K_1$, a $2$-tree, or a hollowed $2$-tree.
If $G$ is $K_1$ or $K_2$, then there is nothing to prove.
Therefore, $G$ is either a cycle or is obtained from a $2$-connected possibly hollowed) $2$-tree
by adding to it a simplicial vertex of degree $2$.
We prove that $G$ has a sink-free $1$-perfect orientation by induction on $|V(G)|$.
If $G$ is a cycle, then $G$ has a sink-free $1$-perfect orientation.
If $G$ is obtained from a $2$-connected (possibly hollowed) $2$-tree $G'$
by adding to it a simplicial vertex, say $v$, of degree $2$, then
the inductive hypothesis implies that $G'$ has a sink-free $1$-perfect orientation, say $D'$.
Extending $D'$ by orienting the two arcs incident with $v$ away from $v$ yields a
sink-free $1$-perfect orientation of $G$, as claimed.
\end{proof}

\begin{lemma}\label{lem:K4-minor-free-biconnected-rooted-extension}
For a biconnected $K_4$-minor-free graph $G$, the following statements are equivalent:
\begin{enumerate}
  \item Some rooted extension of $G$ is $1$-perfectly orientable.
  \item All rooted extensions of $G$ are $1$-perfectly orientable.
  \item $G$ is chordal.
  \item $G$ is either $K_1$ or a $2$-tree.
\end{enumerate}
\end{lemma}

\begin{proof}
First, we show the implication $1\Rightarrow 4$.
Suppose that some rooted extension of a biconnected $K_4$-minor-free graph $G$ is $1$-perfectly orientable.
In particular, $G$ is $1$-perfectly orientable, and hence $\{K_{2,3},F_1\}$-induced-minor-free
by Theorem~\ref{prop:1PO}. By Lemma~\ref{lem:cyclically-orientable-2-conn-1-p-o},
$G$ is either $K_1$, a $2$-tree, or a hollowed $2$-tree.
By Lemma~\ref{lem:1-perfect-orientations-2-trees}, $G$ cannot be a hollowed $2$-tree, and hence condition $4$ holds.

Implication $4\Rightarrow 3$ is clear, implication $3\Rightarrow 2$ follows from Lemma~\ref{lem:chordal}, and implication
$2\Rightarrow 1$ is trivial.
\end{proof}

\begin{corollary}\label{cor:rooted-extension-biconnected}
For a biconnected $K_4$-minor-free graph $G$ and $v\in V(G)$, the rooted graph $G^v$ is
$1$-perfectly orientable if and only if
$G$ is either $K_1$ or a $2$-tree.
\end{corollary}

\subsection{The general case}

\vbox{Now we have all the ingredients ready to complete the characterization of $1$-perfectly orientable $K_4$-minor-free graphs. In the following result we will use the following two operations:
\begin{itemize}
\item $(A_1)$: attach a simplicial vertex of degree $1$.
\item $(A_2)$: attach a simplicial vertex of degree $2$ (that is, add a new vertex and connect it by an edge to exactly two vertices of the graph, which are adjacent to each other).
\end{itemize}}

\begin{theorem}\label{thm:k4-minor-free}
Let $G$ be a connected $K_4$-minor-free graph.
Then the following statements are equivalent:
\begin{enumerate}
  \item $G$ is $1$-perfectly orientable.
  \item $G$ is $\{K_{2,3},F_1,F_2\}$-induced-minor-free.
  \item Every block of $G$ is a $2$-tree, except possibly one, which is either $K_1$ or a hollowed $2$-tree.
  \item $G$ can be constructed from either $K_1$ or a cycle by a sequence of operations $(A_1)$ and $(A_2)$.
\end{enumerate}
\end{theorem}

\begin{proof}
Suppose first that $G$ is biconnected.
By Theorem~\ref{prop:1PO}, condition $1$ implies condition $2$.
By Lemma~\ref{lem:K4-minor-free-biconnected}, condition $2$ implies condition $3$, and condition $3$ implies condition $4$.
Suppose now that $G$ can be constructed from either $K_1$ or a cycle by a sequence of operations $(A_1)$ and $(A_2)$.
Since $G$ is biconnected, we may assume that operation $(A_1)$ was never used in the sequence, unless $G$ is isomorphic to $K_2$.
Therefore, $G$ is either $K_1$, $K_2$, or can be constructed from a cycle by a sequence of operations $(A_2)$.
By Lemma~\ref{lem:K4-minor-free-biconnected}, this implies condition $1$.

\begin{sloppypar}
We are left with the case when $G$ has a cut vertex. In this case, we first establish the equivalence of conditions $1$ and $3$.
Let $T$ be the block tree of $G$. By Theorem~\ref{thm:biconnected-reduction}, $G$ is $1$-perfectly orientable~if and only if one of
the following conditions holds:
\begin{itemize}
  \item There exists a block $B_r$ of $G$ such that $B_r$ is $1$-perfectly orientable~and
  for every arc \hbox{$(B,v)\in {\cal B}\times C$} of the $B_r$-rooted orientation of $T$, the rooted graph $B^v$ is $1$-perfectly orientable.
  \item There exists a cut vertex $v_r$ of $G$ such that
  for every arc $(B,v)\in {\cal B}\times C$ of the $v_r$-rooted orientation of $T$, the rooted graph $B^v$ is $1$-perfectly orientable.
\end{itemize}
Since each block of $G$ is a biconnected $K_4$-minor-free graph, Lemma~\ref{lem:K4-minor-free-biconnected}
and~Corollary~\ref{cor:rooted-extension-biconnected} imply that the above two conditions are equivalent, respectively, to the following two:
\begin{itemize}
  \item There exists a block $B_r$ of $G$ such that $B_r$ is either a $2$-tree or a hollowed $2$-tree, and for every arc \hbox{$(B,v)\in {\cal B}\times C$} of the $B_r$-rooted orientation of $T$, the graph $B$ is a $2$-tree.
  \item There exists a cut vertex $v_r$ of $G$ such that for every arc $(B,v)\in {\cal B}\times C$ of the $v_r$-rooted orientation of $T$, the graph $B$ is a $2$-tree.
\end{itemize}
Since the only sink of a $w$-rooted orientation of a tree $T'$ (with $w\in V(T')$) is $w$, the two conditions can be further simplified as follows:
\begin{itemize}
  \item There exists a block $B_r$ of $G$ such that $B_r$ is either a $2$-tree or a hollowed $2$-tree, and every other block $B\neq B_r$ is a $2$-tree.
  \item All blocks of $G$ are $2$-trees.
\end{itemize}
Clearly, one of these two conditions holds if and only if condition 3 holds. This establishes the equivalence of conditions $1$ and $3$.
\end{sloppypar}

The equivalence of conditions $2$ and $3$ follows from Lemma~\ref{lem:K4-minor-free-with-cut-vertex}.
The implication $3\Rightarrow 4$ can be proved by induction on $|V(G)|$, as follows. If $|V(G)| = 1$, then $G$ is isomorphic to $K_1$ and we are done.
Otherwise, $G$ has an end block $B$ that is not a hollowed $2$-tree. Let $v$ be the cut vertex of $G$ contained in $B$. By induction, the graph $G' = G-(V(B)\setminus\{v\})$ can be constructed from either $K_1$ or a cycle by a sequence of operations $(A_1)$ and $(A_2)$. Such a sequence can be extended with an operation of the form $(A_1)$ (resulting in a simplicial vertex $w$ with a unique neighbor $v$) to create a new block corresponding to $B$ and then with a sequence of operations of the form $(A_2)$ to grow $B$ out of the edge $\{v,w\}$.
The implication $4\Rightarrow 1$ can also be proved by induction on the number of vertices,
using the fact that $K_1$ and cycles are $1$-perfectly orientable and that a $1$-perfect orientation of a graph $G$ can be extended to a $1$-perfect orientation of a graph obtained from $G$ by adding to it a simplicial vertex $v$ by orienting
the edges incident with $v$ away from $v$.
\end{proof}

As a consequence of Theorem~\ref{thm:k4-minor-free} we obtain the following result.

\begin{sloppypar}
\begin{corollary}
For every graph $G$, the following statements are equivalent:
\begin{enumerate}
  \item $G$ is $1$-perfectly orientable and $K_4$-minor-free.
  \item $G$ is $1$-perfectly orientable and cyclically orientable.
 \item $G$ is $\{K_4, K_{2,3}, F_1,F_2\}$-induced-minor-free.
\end{enumerate}
\end{corollary}
\end{sloppypar}

\begin{proof}
Since each of the three properties are closed under taking components and disjoint union, we may assume that $G$ is connected.
The equivalence $1\Leftrightarrow 3$ is then an immediate consequence of Theorem~\ref{thm:k4-minor-free}.
The implication $2\Rightarrow 3$ follows from Theorems~\ref{prop:1PO} and~\ref{thm:K4K23-i-m-free}.
The implication $3\Rightarrow 2$ follows from Theorems~\ref{thm:k4-minor-free} and~\ref{prop:1PO}.
\end{proof}


\section{$1$-perfectly orientable outerplanar graphs}\label{sec:outerplanar}

Since every outerplanar graph is $K_4$-minor-free, we can derive from Theorem~\ref{thm:k4-minor-free} a characterization of
$1$-perfectly orientable outerplanar graphs. In the following result we will use the following two operations:
\begin{itemize}
\item $(A_1)$ attach a simplicial vertex of degree $1$.
\item $(A_2')$ attach a simplicial vertex of degree $2$ to adjacent vertices $v$ and $w$ where the edge $vw$ lies in at most one induced cycle.
\end{itemize}

\begin{theorem}\label{thm:outerplanar-characterization}
For a connected outerplanar graph $G$, the following statements are equivalent:
\begin{enumerate}
  \item $G$ is $1$-perfectly orientable.
  \item $G$ is $\{K_{2,3},F_1,F_2\}$-induced-minor-free.
  \item Every block of $G$ is a $2$-tree, except possibly one, which is either $K_1$ or a hollowed $2$-tree.
  \item $G$ can be constructed from either $K_1$ or a cycle by a sequence of operations $(A_1)$ and $(A_2')$.
\end{enumerate}
\end{theorem}

\begin{proof}
The equivalences $1\Leftrightarrow 2 \Leftrightarrow 3$ as well as the implication $4\Rightarrow 1$ follow from Theorem~\ref{thm:k4-minor-free}. From Theorem~\ref{thm:k4-minor-free} we also know that if one of conditions $1$, $2$, or $3$ holds, then $G$ can be constructed from either $K_1$ or a cycle by a sequence of operations $(A_1)$ and $(A_2)$.
Suppose that, when using the operation $(A_2)$ to add a simplicial vertex $u$ with neighbors $v$ and $w$, the edge $vw$ already lies in two (distinct) induced cycles,  say $C$ and $C'$. First, we claim that $C$ and $C'$ intersect in a path (which contains the edge $vw$). Suppose that this is not the case. Then the intersection of $C$ and $C'$ consist of at least two components, each of which is a path. Let $x$ be  an endvertex of one of these path components, and let $P$ be the path in $C'$ with $x$ as an endvertex, whose internal vertices and edges are not in $C$, and the other endvertex is $y\in V(C)\cap V(C')$. Now, it is easy to see that the subgraph induced by $V(C)\cup V(P)$ contains $K_{2,3}$ as a minor; this implies that the graph is not outerplanar, and since this property is preserved in further steps of the procedure, this contradicts the assumption that $G$ is outerplanar. Thus $C$ and $C'$ intersect in a path, which contains $vw$. If this path contains other vertices than $v$ and $w$, then one can again easily derive that $K_{2,3}$ appears as a minor, contradicting outerplanarity of $G$. Hence $C$ and $C'$ intersect exactly in the subgraph $K_2$ formed by $v$ and $w$. Then, after applying the operation $(A_2)$ of adding the vertex $u$ as the neighbor of $v$ and $w$, we infer that the sets $\{u\}$, $\{v\}$, $\{w\}$, $V(C) \setminus \{v,w\}$, and $V(C')\setminus \{v,w\}$ form an induced minor model of $K_{2,3}^{+}$. Hence, in this case the obtained graph would not be outerplanar, and it would remain non-outerplanar until the end of the procedure. Therefore, we deduce that in each step of the construction that uses $(A_2)$, in fact an operation is of the form $(A_2')$. This proves the implication $1\Rightarrow 4$.

\end{proof}

As a consequence of Proposition~\ref{prop:outerplanar} and Theorem~\ref{thm:outerplanar-characterization} we obtain the following result.

\begin{sloppypar}
\begin{corollary}
For every graph $G$, the following statements are equivalent:
\begin{enumerate}
  \item $G$ is $1$-perfectly orientable and outerplanar.
  \item $G$ is $\{K_4, K_{2,3}, K_{2,3}^{+}, F_1,F_2\}$-induced-minor-free.
\end{enumerate}
\end{corollary}
\end{sloppypar}

\begin{proof}
Let $G$ be a $1$-perfectly orientable outerplanar graph. Then $G$ is $\{K_4, K_{2,3}, K_{2,3}^{+}\}$-induced-minor-free since $G$ is outerplanar, and $\{F_1, F_2\}$-induced-minor-free since it is $1$-perfectly orientable. Conversely, if $G$ is $\{K_4, K_{2,3}, K_{2,3}^{+}, F_1,F_2\}$-induced-minor-free then $G$ is outerplanar by Proposition~\ref{prop:outerplanar}. By Theorem~\ref{thm:outerplanar-characterization}, $G$ is also $1$-perfectly orientable.
\end{proof}

\section{Concluding remarks}\label{sec:conclusion}

Since outerplanar and $K_4$-minor-free graphs are subclasses of the class of planar graphs (see, e.g.,~\cite{MR1686154}), it is 
a natural question whether the characterizations of $1$-perfectly orientable graphs within these two graph classes
given by Theorems~\ref{thm:k4-minor-free} and~\ref{thm:outerplanar-characterization} could be generalized to the class of planar graphs.
While no such characterizations are presently known, we observe below that known results on treewidth imply a partial result in this direction, namely that $1$-perfectly orientable planar graphs are of bounded treewidth. We remind the reader that every outerplanar graph is of treewitdh at most $2$ and, more generally, $K_4$-minor-free graphs are exactly the graphs of treewidth at most $2$.

A {\it $k\times k$} grid is the graph with vertex set $\{1,\ldots, k\}^2$ and edge set $\{\{(i,j),(i',j')\}: 1\le i,j,i',j'\le k, |i-i'|+|j-j'| = 1\}$.
One of the results from the graph minor project states that for every positive integer $k$ there is a positive integer $N$ such that if $G$ is a graph of treewidth at least $N$ then the $k\times k$ grid is a minor of $G$. This result was further strengthened for planar graphs in several ways.
For example, a result due to Gu and Tamaki~\cite{MR2965273} implies the following.

\begin{proposition}\label{prop:tw}
For every planar graph $G$, the treewidth of $G$ is at most $4.5k-1$ where $k$ is the largest integer such that $G$
contains a $k\times k$ grid as a minor.
\end{proposition}

\begin{proof}
Let $G$ be a planar graph, let $k$ be as above, and let $b$ and $t$ denote the treewidth and the branchwitdh of $G$, respectively.
Since $G$ is planar, a result by Gu and Tamaki~\cite{MR2965273} implies that $b\le 3k$.
Moreover, we have $t\le \max\{1.5b-1,1\}$ by a general result relating the treewidth and the branchwidth due to
Robertson and Seymour~\cite{MR1110468}. Consequently, $t\le \max\{4.5k-1,1\} = 4.5k-1$ since $k\ge 1$.
\end{proof}

\begin{corollary}\label{cor:grid}
For every $k>1$, the treewidth of every planar graph having no $k\times k$ grid as minor is at most $4.5(k-1)-1$.
\end{corollary}

Next, observe that a minor model of a $6\times 6$ grid in a planar graph $G$ can be used to obtain an induced minor model of $F_1$ in $G$ (see Fig.~\ref{fig:F1-in-grid}).

\begin{figure}[h!]
\begin{center}
\includegraphics[height=40mm]{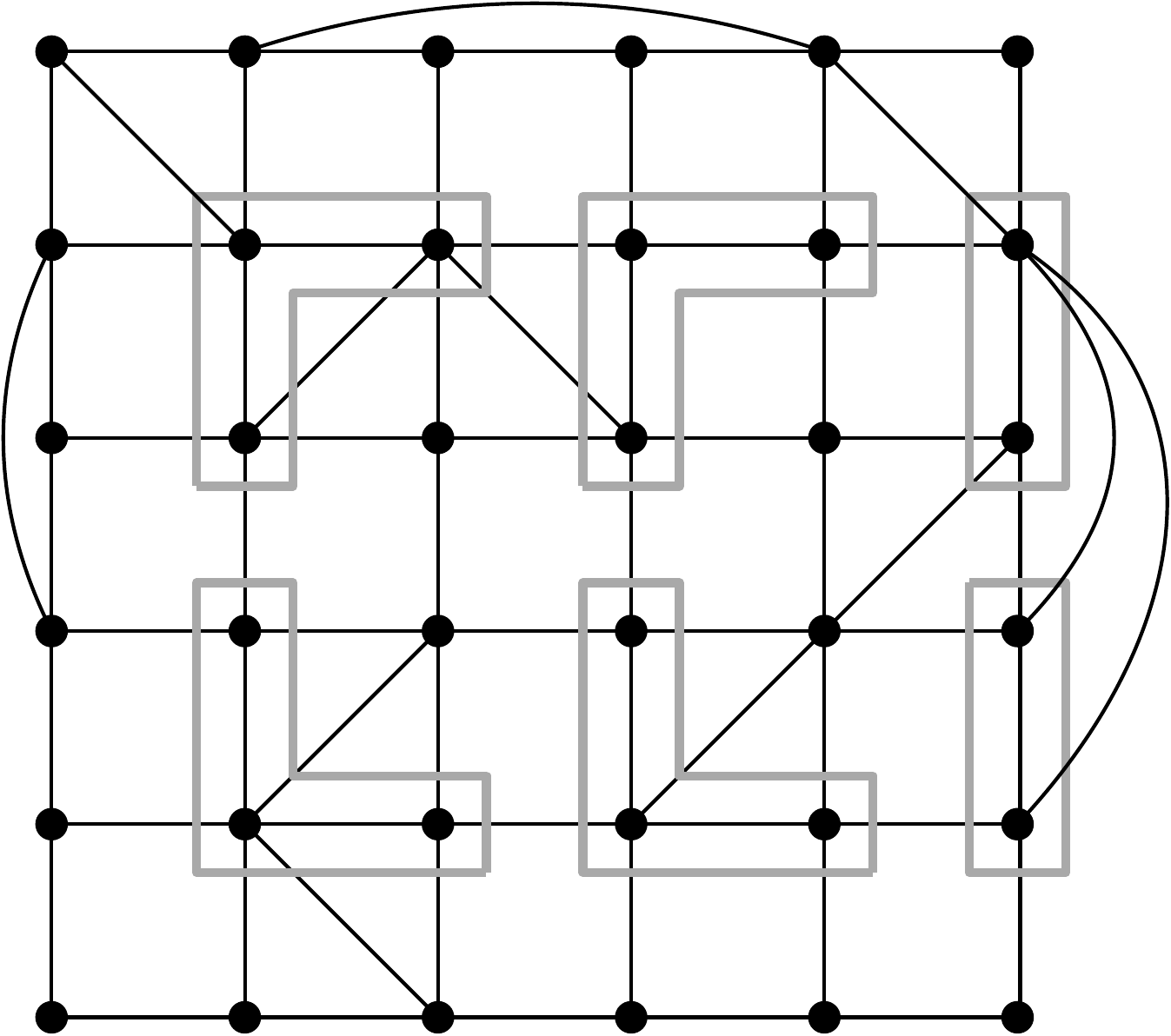}
\caption{Obtaining $F_1$ as induced minor in a planar graph having the $6\times 6$ grid as a minor.}\label{fig:F1-in-grid}
\end{center}
\end{figure}

Therefore, no $1$-perfectly orientable planar graph can have a $6\times 6$ grid as minor and Corollary~\ref{cor:grid} implies the following.

\begin{corollary}\label{cor:planar-1-po}
The treewidth of every $1$-perfectly orientable planar graph is at most $21$.
\end{corollary}

More generally, for every positive integer $r$, the treewidth is bounded in the class of $1$-perfectly orientable $K_r$-minor-free graphs.
This follows from the analogous statement in the more general setting, for $K_r$-minor-free graphs excluding any fixed planar graph as induced minor,
which can be proved using arguments as in Case 2 of the proof of \cite[Theorem 9]{MR2901091}. (Theorem 9 from~\cite{MR2901091} was
derived from results due to Fellows et al.~\cite{MR1308575} and Fomin et al.~\cite{MR2557796}.)
This observation has the following algorithmic consequence: since the defining property of $1$-perfectly orientable graphs can be expressed in Monadic Second Order Logic with quantifiers over edges and edge subsets, Courcelle's Theorem~\cite{MR1042649} implies that  $1$-perfectly orientable graphs can be recognized in linear time in any class of graphs of bounded treewidth. In particular, by the above observation, this is the case for any class of graphs excluding some complete graph as a minor.

We conclude with a question 
that could lead to further insights on the structure of $1$-perfectly orientable graphs;
a positive answer would generalize Corollary~\ref{cor:planar-1-po}.

\begin{problem}
Is it true that for every positive integer $k$ there is a positive integer $N$ such that
every $1$-perfectly orientable graph with clique number $k$ is of treewidth at most $N$?
\end{problem}

\section*{Acknowledgements}

We are grateful to Marcin J.~Kami\'nski for pointing us to the paper~\cite{MR2901091}.
This work was supported in part by the Slovenian Research Agency (I$0$-$0035$, research programs P$1$-$0285$ and P$1$-$0297$, research projects N$1$-$0032$, N$1$-$0043$, J$1$-$5433$, J$1$-$6720$, J$1$-$6743$, J$1$-$7051$, and J$1$-$7110$, and two Young Researchers Grants.)

\bibliographystyle{abbrv}

\begin{sloppypar}
\bibliography{1-perfectly-orientable-bib}{}
\end{sloppypar}

\end{document}